\documentclass[12pt,a4paper]{article}
\usepackage{amsmath,amssymb,amsthm,hyperref}
\usepackage{graphicx}
\usepackage{tikz}
\usepackage{version} % include, exclude and comment
\definecolor{refkey}{rgb}{0.65,0.38,0.15}
\usepackage{enumerate}
\usepackage{datetime}
\numberwithin{equation}{section}
\usepackage[bf]{caption}     % captionstyles
\usepackage{geometry}        % for good layout
\usepackage{bm}  % bold symbols in math mode \boldsymbol{\alpha}
\usepackage[utf8]{inputenc}

\usepackage{buczobhmvlipchar}
  \newtheorem{theorem}{Theorem}[section]
 \newtheorem{corollary}[theorem]{Corollary}
 
 \newtheorem{lemma}[theorem]{Lemma}

 {\theoremstyle{definition}
 \newtheorem{definition}[theorem]{Definition}
 
 \newtheorem{notation}[theorem]{Notation}

 }
\newcommand{\N}{\ensuremath{\mathbb N}} %natural numbers
\newcommand{\R}{\ensuremath{\mathbb R}} %real numbers
\newcommand{\beq}{\begin{equation}}
\newcommand{\eeq}{\end{equation}}
\newcommand{\be}{\begin{enumerate}}
\newcommand{\ee}{\end{enumerate}}
\newcommand{\bi}{\begin{itemize}}
\newcommand{\ei}{\end{itemize}}

\DeclareMathOperator{\lip}{lip\kern-0.8pt}
\DeclareMathOperator{\Lip}{Lip\kern-0.8pt}

\definecolor{skyblue}{rgb}{0,0.4,0.6}
\definecolor{red}{rgb}{0.6,0,0}
\definecolor{green}{rgb}{0,0.6,0}
\definecolor{aquam}{rgb}{0.5,1.0,1.0}
\definecolor{bbrown}{rgb}{0.75,0.38,0.15}
\definecolor{Cyan}{rgb}{0,0.6,0.6}
\definecolor{Darkblue}{rgb}{0,0,0.9}
\definecolor{Dodgerblue2}{rgb}{0,0.5,1}
\definecolor{Green}{rgb}{0,0.5,0.1}
\definecolor{dGreen}{rgb}{0,0.5,0}
\definecolor{Kahki}{rgb}{1,1,0.5}
\definecolor{Magenta}{rgb}{1,0,1}
\definecolor{bMagenta}{rgb}{1,.6,1}
\definecolor{Orange}{rgb}{0.8,0.3,0}
\definecolor{dOrchid}{rgb}{0.7,0.2,0.4}
\definecolor{Orchid}{rgb}{1,0.5,1}
\definecolor{Purple}{rgb}{0.65,0.07,0.85}
\definecolor{Royalblue}{rgb}{0.6,0.85,0.87}
\definecolor{Tan}{rgb}{0.54,0.42,0.23}
\definecolor{bTan}{rgb}{0.94,0.82,0.63}
\definecolor{Turquoise}{rgb}{0,0.85,0.87}
\definecolor{Yellow}{rgb}{1,1,0}
\definecolor{bYellow}{rgb}{1,1,0.6}
\definecolor{bRed}{rgb}{1,0.7,0.7}
\definecolor{dRed}{rgb}{0.7,0,0}
\definecolor{dRed}{rgb}{1,0,0}
\definecolor{boxcolb}{rgb}{0.87,0.77,0.75}%rosybrown
\definecolor{boxcol}{rgb}{0.6,0.85,0.87}%cadetblue
\definecolor{boxcolgreen}{rgb}{0.64,0.93,0.79}
\definecolor{boxcolaa}{rgb}{.75,.99,.70}
\definecolor{boxcolbb}{rgb}{0.39,0.50,0.56}
\definecolor{boxcolcc}{rgb}{1,0.81,0.65}
\definecolor{yy}{rgb}{0.43,0.21,.18}
\definecolor{gA}{gray}{0.5}
\definecolor{gB}{gray}{0.8}
\definecolor{gC}{gray}{0.9}
\title{Characterization of $\lip$ sets}
\author{Zolt\'an Buczolich\thanks{\scriptsize
This author was supported by the Hungarian National Research, Development and Innovation Office--NKFIH, Grant 124003.
},
Department of Analysis, ELTE E\"otv\"os Lor\'and\\
University, P\'azm\'any P\'eter S\'et\'any 1/c, 1117 Budapest, Hungary\\
email: buczo@caesar.elte.hu\\
{\tt  http://buczo.web.elte.hu}\\
ORCID Id: 0000-0001-5481-8797
  \smallskip\\
  Bruce Hanson, Department of Mathematics,\\ Statistics and Computer Science,\\ St.\ Olaf College,
Northfield, Minnesota 55057, USA\\
{email:} \texttt{hansonb@stolaf.edu}
 \smallskip\\
 Bal\'azs Maga\thanks{\scriptsize This author was supported by the \'UNKP-19-3 New National Excellence of the Hungarian Ministry of Human Capacities, and by the Hungarian National Research, Development and Innovation Office–NKFIH, Grant 124749.},
Department of Analysis, ELTE E\"otv\"os Lor\'and\\
University, P\'azm\'any P\'eter S\'et\'any 1/c, 1117 Budapest, Hungary\\
 email: magab@caesar.elte.hu \\{\tt  http://magab.web.elte.hu/}
  \smallskip\\
and
  \smallskip\\
 G\'asp\'ar V\'ertesy\thanks{\scriptsize This author was supported by the Hungarian National Research, Development and Innovation Office–NKFIH, Grant 124749.
},
 Department of Analysis, ELTE E\"otv\"os Lor\'and\\
University, P\'azm\'any P\'eter S\'et\'any 1/c, 1117 Budapest, Hungary\\
email: vertesy.gaspar@gmail.com\
}
\date{\today}
\begin{document}
\maketitle

\newpage

\begin{abstract}
We denote the local ``little" Lipschitz constant of a function
$f: {\ensuremath {\mathbb R}}\to {\ensuremath {\mathbb R}}$ by $ \lip f$. In this paper we settle the following question: For which sets $E {\subset}  {\ensuremath {\mathbb R}}$ is it possible to find a continuous function $f$ such that
$ \lip f=\mathbf{1}_E$?

In an earlier paper we introduced the concept of strongly one-sided dense sets. Our main result characterizes $\lip 1$ sets as countable unions of closed sets which are strongly one-sided dense.

We also show that a stronger statement is not true i.e. there are strongly one-sided dense $F_\sigma$ sets which are not $\lip 1$.
  
%We also show that there are non-empty closed sets which contain no non-empty closed subsets which are strongly one-sided dense.
   \end{abstract}

\section{Introduction}\label{*secintro}

We begin by introducing some basic notation.
We will assume throughout that $f: {\mathbb R} \to  {\mathbb R}$ is continuous.
Then the so-called ``big Lip'' and ``little lip'' functions are defined as follows:

 \begin{equation}
  \Lip
  f(x)= \limsup_{r\to 0^+}M_f(x,r),\qquad\label{Lipdef}
  \lip
  f(x)= \liminf_{r\rightarrow 0^+}M_f(x,r),
 \end{equation}
where
$$M_f(x,r)=\frac{\sup\{|f(x)-f(y)| \colon |x-y| \le r\}}r.$$

The origin of the big Lip function dates back to the early 1900s, while the little lip function is a more recent phenomenon.  As far as we know, it appears for the first time in a paper by Balogh and Cs\"ornyei \cite{[BaloghCsornyei]}.  More recently, there have been a number of papers dealing with various aspects of the little lip function.  See \cite{[Hanson]}, \cite{[Hanson2]}, \cite{[BHRZ]}, 
\cite{[MaZi]} and \cite{[Zi]}. 

In \cite{[BHMVlip]}, the authors of this note investigated when it is possible for $ \Lip f$ (or $ \lip f$) to be a characteristic function.  To expedite this investigation we set the following definition:  Given a set $E \subset  {\mathbb R}$ we say that $E$ is  $ \Lip 1$ ($ \lip 1$) if there is a continuous function defined on $ {\mathbb R}$ such that $ \Lip f =\mathbf{1}_E$, ($ \lip f=\mathbf{1}_E$).  The main results in \cite{[BHMVlip]} gave necessary or sufficient conditions for $E$ to be $ \Lip 1$ or $ \lip 1$.  We were not able to come up with a characterization of either type of set. 

Our main result in this note (presented in Section \ref{*seclip}) is to improve on {Theorems 4.7 and 4.8} from \cite{[BHMVlip]} by characterizing $\lip 1$ sets as countable unions of closed sets satisfying the following density property:

\begin{definition}
The set $E$ is {\it strongly one-sided dense} at $x$ if for any sequence $\{I_n\}=\{[x-r_n,x+r_n]\}$ such that
$r_n \to 0^+$  we have 
$$\max\Big \{\frac{|E \cap [x-r_n,x]|}{r_n},\frac{|E \cap [x,x+r_n]|}{r_n}\Big \}\to 1.$$
(Here and elsewhere in this paper $|E|$ denotes the Lebesgue measure of the set $E$.)
The set $E$ is {\it strongly one-sided dense} (SOSD) if $E$ is strongly one-sided dense at every point $x \in E$.
\end{definition}

 Quite often obtaining a result for the $\lip$ exponent is more difficult than deducing a corresponding result for the $\Lip$ one, therefore it is a bit peculiar that the question of a similar characterization of $\Lip 1$ sets remains open. In this direction the following are known: in Theorem 4.1 of \cite{[BHMVlip]} we showed that if $E\subset {\ensuremath {\mathbb R}}$ is $ \Lip 1$ then $E$ is a weakly dense $G_\delta$ set.
 (For the definition of weak density we refer to \cite{[BHMVlip]}.)
 However, in Theorem 6.3 of the same paper we showed that there exists a weakly dense, $G_\delta$ set $E\subset {\ensuremath {\mathbb R}}$ which is not $ \Lip 1$, thus this condition on $E$ is necessary, but not sufficient.

It is worth mentioning that at first glance one might believe that every SOSD  $F_\sigma$ set can be written as the countable union of  SOSD  closed sets, implying that such sets are $\lip 1$ due to our characterization. If { this} were true, our theorem could be formulated more neatly by saying that the $\lip 1$ sets are precisely the  SOSD  $F_\sigma$ sets.
%In \cite[Theorem 4.7]{[BHMVlip]} it was shown that $\lip 1$ sets are strongly one-sided dense and $F_\sigma$ 
%(this is an obvious consequence of Theorem \ref{char_thm} which is the main result of the present paper). 
However, in Section \ref{not lip 1} we will show that the above intuition is misleading: %the converse is not true: 
there is  an  SOSD  $F_\sigma$ set which does not contain any nonempty, closed,  SOSD  subsets, {and} therefore is not expressible as a union of such sets.

%In measure theory quite often one can verify that a set differs only in a small, measure zero set from some much nicer set.
%For example in Theorem 1.1 of \cite{[BHMVlipap]} we showed that for every measurable set $E$ there exists a $G_{\delta}$, $\Lip 1$ set $\widetilde{E}$ such that $|E\triangle \widetilde{E}| = 0$.  
%First one might think that such approximation is also possible by strongly one-sided dense sets. 
%Theorem \ref{bad SOSD} in Section \ref{not lip 1} shows that it is not the case, we construct a non-empty closed set that fails to have a non-empty subset which is closed and strongly one-sided dense.

\section{Characterizing little lip sets}\label{*seclip}

\begin{comment}
	
\begin{notation}
For every $A\subset\R$ let $A^c:=\R\setminus A$.
\end{notation}

\end{comment}

\begin{notation}
 For any $S,T\subset  {\mathbb R}$ we define $d(S,T)$ to be the distance from $S$ to $T$,
that is $\inf\{|x-y|:x\in S,\  y\in T \}$. Moreover, for any $x \in  {\mathbb R}$, simply put $d(x,S)=d(\{x\},S)$.
\end{notation}

%\begin{definition}\label{l converge}
%We write $I_n \stackrel{l}{\to} x$ (resp.~$I_n \stackrel{r} \to x$) if $(I_n)$ is a sequence of closed %intervals with $I_n=[x-r_n,x]$ (resp.~$I_n=[x,x+r_n]$) and $r_n \searrow 0$.
%\end{definition}

%\begin{definition}\label{one side dense}
%The set $E$ is {\it right} ({\it left}) {\it dense} at $x$ if for any sequence $(I_n)$ such that
%$I_n \stackrel{r}{\to} x$ ($I_n \stackrel{l}{\to}x$) we have $\frac{|E \cap I_n|}{|I_n|}\to 1$. In this %case, we say that $x$ is a {\it right density point} ({\it left density point}) of $E$.
%The set $E$ is {\it one-sided dense} if $E$ is either right or left dense at every point $x \in E$.
%\end{definition}

As noted in the introduction our main result is the following:

\begin{theorem}\label{char_thm}
The set $E \subset \R$ is $\lip 1$ if and only if $E=\Union_{n=1}^\infty E_n$ where each $E_n\subset\R$ is a strongly one-sided dense closed set.
\end{theorem}

\begin{proof}
We begin by proving the sufficiency condition so assume that  $E=\Union_{n=1}^\infty E_n$, where each $E_n$ is closed and SOSD.  We may assume without loss of generality that $E_1\subset E_2 \subset E_3 \subset\ldots$ and $E_1\neq\emptyset$. 
Let $E_0:=\emptyset$.

%For every $n\in\N$ set $G_n := E_n\setminus \Union_{k=1}^{n-1} F_k$.

Set $f_1\colon\R\to\R$ such that 
$$
f_1(x)=
\begin{cases}
|[0,x]\cap E_1| & \text{if } x \ge 0 \\
|[x,0]\cap E_1| & \text{if } x < 0
 .   \end{cases}
$$ 

Let $n>1$.
We will define $f_n\colon\R\to\R$ to satisfy
\begin{equation}\label{szuk}
0 \le f_n(x) \le \min\{2^{-n},2^{-n}\cdot d^2(x,E_{n-1})\}
\end{equation}
for every $x\in\R$.   
 Choose an interval $I=(a,b)$,   or a half-line     contiguous to $E_{n-1}$. 
Suppose that $a$ is finite (the other case is similar).
Define a sequence $(a_k)_{k=0}^\infty$ in $(a,b)$ for which 
\begin{enumerate}[(I)]
\item $a_0=(a+b)/2$ if $b<\infty$, and $a_0=a+1$ if $b=\infty$,
\item\label{a_k suruseg} $a_{k-1}-a_k < \min\{2^{-n}(a_k-a)^2,2^{-n}\}$,
\item $a_k\searrow a$.
\end{enumerate}
If $x\in (a_{k},a_{k-1}]$ for some $k\in\N$, set
\begin{equation}\label{E_n def}
f_n(x) := 
\begin{cases}
|(a_k,x)\cap E_n| & \text{if } |(a_k,x)\cap E_n| \le |(x,a_{k-1})\cap E_n| \\
|(x,a_{k-1})\cap E_n| & \text{if } |(a_k,x)\cap E_n|>|(x,a_{k-1})\cap E_n| 
 .   \end{cases}
\end{equation}
If $b<\infty$ set $f_n$ similarly on $[(a+b)/2,b)$. 
If $b=\infty$, $k\in\N$ and $x\in \big(a_0+(k-1)2^{-n},a_0+(k2^{-n})\big)$, let
$$
f_n(x) := 
\begin{cases}
|(a_0+(k-1)2^{-n},x)\cap E_n| & \text{if } {\substack{\textstyle |(a_0+(k-1)2^{-n},x)\cap E_n|\\ \textstyle \le |(x,a_0+k2^{-n})\cap E_n|}} \\
|(x,a_0+k2^{-n})\cap E_n| & \text{if } {\substack{\textstyle |(a_0+(k-1)2^{-n},x)\cap E_n| \\ \textstyle > |(x,a_0+k2^{-n})\cap E_n|.}}
\end{cases}
$$
Let $f_n|_{E_{n-1}} :\equiv 0$.
Observe that for every $x\in\R$ and $y\in E_{n-1}$ there is a $y'$ in the closed interval determined by $x$ and $y$ such that $f_n(y')=0$ and $|x-y'|\le \min\{2^{-n},2^{-n}|x-y|^2\}$, which implies \eqref{szuk}.
%By \eqref{a_k suruseg} and \eqref{E_n def}, if $x\in E_{n-1}$ and $y\in\R$, then 
%\begin{equation}\label{lapos}
%|f_n(x)-E_n(y)|\le \max\{2^{-n},2^{-n}(x-y)^2\}.
%\end{equation}
%Thus, if $x\in E_{n-1}$, then for every $y\in\R$

Define $f\colon\R\to\R$ by $f(x):=\sum_{n=1}^\infty f_n(x)$. 
If $x,y\in\R$ and $x<y$ we have that 
$$
\frac{|f(x)-f(y)|}{y-x} 
\le \sum_{n=1}^\infty \frac{|f_n(x)-f_n(y)|}{y-x} 
$$
$$\le \sum_{n=1}^\infty \frac{|(E_n\setminus E_{n-1})\cap (x,y)|}{y-x} 
= \frac{|E\cap (x,y)|}{y-x} 
\le 1,
$$
hence $\lip f(x)\le 1$ for every $x\in\R$.

Suppose that $x\in E_{n_x}\setminus E_{n_x-1}$ for some $n_x\in\N$. 
Since $E_{n_x}$ is SOSD and $E_{n_x-1}$ is closed, for every $\eps>0$ there is an $r_x>0$ such that for every $r\in (0,r_x)$ 
$$
\max\{|(E_{n_x}\setminus E_{n_x-1})\cap (x,x-r)|,|(E_{n_x}\setminus E_{n_x-1})\cap (x,x+r)|\} > r(1-\eps).
$$
Fix $r \in (0,r_x)$. 
Suppose that $|(E_{n_x}\setminus E_{n_x-1})\cap (x,x+r)| > r(1-\eps)$ (the other case is similar). 
By the definition of the {$f_n$s}, if $r$ is small enough, then %$E_n|_{(x,x+r)}$ is constant for all $n<n_x$, and 
%for every $z\in (x,x+r)$ we have 

$$  
|f_{n_x}(x+r) - f_{n_x}(x)| = |(E_{n_x}\setminus E_{n_x-1})\cap (x,x+r)|.  
$$
Consequently,
\begin{align*}
|f(x)-f(x+r)| 
&\ge |f_{n_x}(x)-f_{n_x}(x+r)|-\sum_{n\in\N\setminus\{ n_x\}} |f_n(x)-f_n(x+r)| \\
&\ge |(E_{n_x}\setminus E_{n_x-1})\cap (x,x+r)| - \sum_{n\in\N\setminus \{n_x\}} |(E_{n}\setminus E_{n-1})\cap (x,x+r)| \\
&\ge r(1-\eps)-r\eps = r(1-2\eps).
\end{align*}
Thus $\lip f(x) \ge 1$.

If $d(x,E)>0$, then there is a neighbourhood $U_x$ of $x$ such that 
$f_n|_{U_x}$ is constant for every $n\in\N$, hence $f|_{U_x}$ is also constant and $\lip f(x) = 0$.

%Suppose that $r>0$, $n\in\N$ and $n>1$. 
%If $(E_n\setminus F_{n-n})\cap (x-r,x+r) = \emptyset$, then $E_n$ is constant on $(x-r,x+r)$. 
%Otherwise, let

If $x\notin E$ and  $d(x,E)=0$, then for every $\eps>0$ there is an $n_\eps\in\N$ for which $E_{n_\eps}\cap (x-\eps,x+\eps) \neq \emptyset$. 
Let $x_\eps\in E_{n_\eps}$ be such that $|x-x_\eps|=d(x,E_{n_\eps})$. 
%and suppose that $x_\eps<x$. 
%As $F_{n_\eps}$ is closed, we can assume that $F_{n_\eps}\cap (x_\eps,x) = \epmtyset$.
Hence, $f_n$ is constant on $(x-|x-x_\eps|,x+|x-x_\eps|)$ for every $n\le n_\eps$.
%Let $r'_x := x-x_\eps$.
By \eqref{szuk}, for every $y\in (x-|x-x_\eps|,x+|x-x_\eps|)$ and $n > n_\eps$
\begin{equation*}%\label{lapos2}
\begin{split}
|f_n(x)-f_n(y)| 
&\le |f_n(x)-f_n(x_\eps)|+|f_n(x_\eps)-f_n(y)| \\
&\le 2^{-n}(x-x_\eps)^2 + 2^{-n} (x_\eps-y)^2 \\
&\le 2^{-n}(x-x_\eps)^2 + 2^{-n} (2(x-x_\eps))^2
\le (2^{-n}+2^{-n+2})\eps |x-x_\eps|.
%\le 2\cdot 2^{-n} (2\eps)^2. 
\end{split}
\end{equation*}
Thus 
\begin{equation*}%\label{lapos2}
\begin{split}
|f(x)-f(y)| 
&\le \sum_{n=n_\eps+1}^\infty  |f_n(x)-f_n(y)|
\le \sum_{n=n_\eps+1}^\infty (2^{-n}+2^{-n+2})\eps |x-x_\eps| \\
&\le \sum_{n=2}^\infty (2^{-n}+2^{-n+2})\eps |x-x_\eps| = \Big( \frac{1}{2}+2\Big)\eps|x-x_\eps| .
%&\le \sum_{n=n_\eps}^\infty 2^{-n}(x-x_\eps')^2 + 2^{-n} (x_\eps-y)^2 
%\le 2\cdot 2^{-n} (2\eps)^2. 
\end{split}
\end{equation*}
Since $\eps$ was chosen arbitrarily, we have that $\lip f(x) = 0$, which concludes the proof of the sufficiency.

For the proof of the necessity we will use the following lemma which is Lemma 4.6 from \cite{[BHMVlip]}.

\begin{lemma}\label{lipnovekedes}
If $E\subset {\ensuremath {\mathbb R}}$, $f\colon {\ensuremath {\mathbb R}}\rightarrow {\ensuremath {\mathbb R}}$ and $ \lip f= \mathbf{1}_E$ 
then $|f(a)-f(b)|\le |[a,b]\cap E|$ for every $a,b\in {\ensuremath {\mathbb R}}$ (where $a<b$).
\end{lemma}

Assume that $E$ is $\lip 1$ and let 
 $f\colon\R\to\R$ be such that $\lip f = \mathbf{1}_E$.
 If $E=\R$, then the proof is trivial so we assume that $E\neq \R$.
Set $G:=\R\setminus E$.
Let $\kappa$ denote the smallest ordinal number for which 
$[1,\kappa)$
has the same cardinality as $G$.
Let $(y_\alpha)_{\alpha\in [1,\kappa)}$ be a well-ordering of $G$. 

 Suppose that $n\in\N$.  %be be fixed until the $\clubsuit$ sign.
We will define $r_{n,y_\alpha},r'_{n,y_\alpha}>0$ for every $\alpha\in [1,\kappa)$ by transfinite recursion on $\alpha$ such that 
\begin{enumerate}[(a)]
\item\label{r kicsi} $r'_{n,y_\alpha}\in (0,n^{-1}),$
\item\label{lapos 0} $M_f(y_\alpha,2r'_{n,y_\alpha})<10^{-2},$
\item\label{jo vegpontok} $r_{n,y_\alpha}\in (0.5r'_{n,y_\alpha},r'_{n,y_\alpha})$ and  $f$  is differentiable at $y_\alpha-r_{n,y_\alpha}$ and $y_\alpha+r_{n,y_\alpha}$,  
\item\label{egy oldalrol} $y_\alpha+r_{n,y_\alpha} \neq y_\beta-r_{n,y_\beta}$ and $y_\alpha-r_{n,y_\alpha} \neq y_\beta+r_{n,y_\beta}$ for every $\beta\in [1,\alpha)$, 
\item\label{resze} if $\alpha>1$ and $y_\alpha\in\Union_{\beta\in [1,\alpha)} (y_\beta-r_{n,y_\beta},y_\beta+r_{n,y_\beta})$ then
$$
(y_\alpha-r'_{n,y_\alpha},y_\alpha+r'_{n,y_\alpha})\subset \Union_{\beta\in [1,\alpha)} (y_\beta-r_{n,y_\beta},y_\beta+r_{n,y_\beta}).
$$
\end{enumerate}
Since $\lip f(y_1)=0$, we can choose $r'_{n,y_1}>0$ to satisfy \eqref{r kicsi} and \eqref{lapos 0}. 
By Lemma \ref{lipnovekedes}, $f$ is Lipschitz  and {therefore} is differentiable at almost every point, 
hence there is an $r_{n,y_1}>0$ such that \eqref{jo vegpontok} holds for $\alpha=1$,   and  conditions \eqref{egy oldalrol} and \eqref{resze} are empty at this step.  
Suppose that $\alpha \in (1,\kappa)$ and we have already defined $r'_{n,y_\beta}$ and $r_{n,y_\beta}$ for every $\beta\in [1,\alpha)$. %with the above properties. 
Take an $r'_{n,y_\alpha}>0$ which satisfies \eqref{r kicsi}, \eqref{lapos 0} and \eqref{resze}. 
Since $f$ is Lipschitz, we can choose $r_{n,y_\alpha}>0$ to make \eqref{jo vegpontok}
   and \eqref{egy oldalrol} 
   true (since the cardinality of $\alpha$ is less than  the cardinality of the  continuum).

Let $x\in G$. We have that,
\begin{equation}\label{M_f(x,2r_{n,x})}
\begin{split}
M_f(x,2r_{n,x})
&= \frac{1}{2r_{n,x}}\cdot \sup \Big\{|f(x)-f(z)| : z\in (x-2r_{n,x},x+2r_{n,x})\Big\} \\
&\le \frac{1}{2r_{n,x}}\cdot \sup \Big\{|f(x)-f(z)| : z\in (x-2r'_{n,x},x+2r'_{n,x})\Big\} \\
&= \frac{1}{2r_{n,x}}\cdot 2r'_{n,x} M_f(x,2r'_{n,x}) 
\underset{\text{by \eqref{jo vegpontok}}}\le 2\cdot M_f(x,2r'_{n,x}) 
\underset{\text{by \eqref{lapos 0}}}\le 2\cdot 10^{-2}
= \frac{1}{50},
\end{split}
\end{equation}
hence for every $x_0,x_1\in (x-2r_{n,x},x+2r_{n,x})$
\begin{equation}\label{x_1 x_2}
|f(x_0)-f(x_1)|\le |f(x_0)-f(x)|+|f(x)-f(x_1)| %\underset{\text{by \eqref{lapos 0}}}{\le}
\le 2\cdot 2r_{n,x} M_f(x,2r_{n,x})
\le \frac{4r_{n,x}}{50}.
\end{equation}
We obtain that for every $x_0\in(x-1.5r_{n,x},x+1.5r_{n,x})$ 
\begin{equation}\label{lapos}
\begin{split}
M_f(x_0,0.5r_{n,x}) 
&\le \sup \Big\{\frac{|f(x_0)-f(x_1)|}{0.5r_{n,x}} : x_1\in (x-2r_{n,x},x+2r_{n,x})\Big\} \\
&\underset{\text{by \eqref{x_1 x_2}}} \le  \frac{4r_{n,x}}{50} \cdot \frac{1}{0.5r_{n,x}} 
= \frac{8}{50}.
\end{split}
\end{equation}

For every $n\in\N$ set $G_n := \Union_{x\in G} (x-r_{n,x},x+r_{n,x})$ and $E_n := \R\setminus G_n$. Hence the sets $E_n$ are closed.

As $G\subset G_n$ for all $n\in\N$, we have that $\Union_{n=1}^\infty E_n\subset E$.

If $x_0\in E$, then there is a $\varrho>0$ such that for every $r\in (0,\varrho)$ we have $M_f(x_0,r)>\frac{8}{50}$.
Thus, if $x_0\in (x-r_{n^*,x},x+r_{n^*,x})$ for some $x\in G$ and $n^*\in\N$, then $0.5r_{n^*,x}\ge \varrho$ by \eqref{lapos}.
As 
\begin{equation}\label{r_{n,x} kicsi}
r_{n^*,x} \underset{\text{by \eqref{jo vegpontok}}}\le r'_{n^*,x} \underset{\text{by \eqref{r kicsi}}}\le \frac{1}{n^*},
\end{equation}
$n^*\le r_{n^*,x}^{-1} \le 2\varrho^{-1}$. Hence $x_0\notin \bigcap_{n=1}^\infty G_n$, that is $x_0\in \Union_{n=1}^\infty E_n$. This implies $E = \Union_{n=1}^\infty E_n$.

Assume that $n\in\N$ is fixed until the end of the proof.

We need a lemma to prove that $E_n$ is SOSD.

\begin{lemma}\label{char-ban}
Let $x_0\in E_n$.
If $x_0\neq x+r_{n,x}$ for every $x\in G$, then for small enough $r>0$
\begin{equation}\label{bal_becsles}
|f(x_0)-f(x_0-r)|
\le r -\frac{21}{400}|G_n\cap (x_0-r,x_0)|.
\end{equation}
Similarly, if $x_0\neq x-r_{n,x}$ for every $x\in G$, 
then for small enough $r>0$
\begin{equation}\label{jobb_becsles}
|f(x_0)-f(x_0+r)|
\le r -\frac{21}{400}|G_n\cap (x_0,x_0+r)|.
\end{equation}
\end{lemma}

Lemma \ref{char-ban} will be proved later.

{ Assume that $x_0\in E_n$. Thus $\lip f(x_0) = 1$.

By Lemma \ref{char-ban}, if $x_0 \neq x+r_{n,x}$ and $x_0 \neq x-r_{n,x}$ for every $x\in G$, then $E_n =\R\sm G_n$ must be SOSD at $x_0$.}

Now suppose that $x_0=x+r_{n,x}$ for some $x\in G$ (the $x_0=x-r_{n,x}$ case is similar).  
By \eqref{egy oldalrol}, we have that $x_0\neq x'-r_{n,x'}$ for every $x'\in G$.  
According to \eqref{jo vegpontok}, $f$ is differentiable at $x+r_{n,x}=x_0$. 
Therefore  $$ \lim_{r\to 0^+}\frac{|f(x_0+r)-f(x_0)|}r=|{f'(x_0)}|=\lip f(x_0) = 1,$$ and hence, the set $E_n$ must be dense (in the classical one-sided Lebesgue density sense) at $x_0$ from the right by \eqref{jobb_becsles}.  
Thus, $E_n$ is SOSD at $x_0$, which concludes the proof of the theorem. \qed

{\bf Proof of Lemma \ref{char-ban}.}
We will prove only \eqref{bal_becsles}, the proof of \eqref{jobb_becsles} is similar. Thus, suppose that  $x_0\in E_n$ and
\begin{equation}\label{njv}
\text{$x_0\neq x+r_{n,x}$ for every $x\in G$.}
\end{equation}

Since {$x_{0}\in E_n \subset E$} we have $\lip f(x_{0})=1$.  
By \eqref{lapos}, we can take an $R_0>0$ such that %for every $r\in (0,R_0)$
if $x\in G \cap (x_0-R_0,x_0)$ %and $x_0-r\in (x-r_{n,x},x+r_{n,x})$ 
then $x_0\notin (x-1.5r_{n,x},x+1.5r_{n,x})$.

We  claim  that there is an $R_1\in (0,R_0)$  such  that 
\begin{equation}\label{R_1}
\begin{gathered}
(x_0-R_1,x_0)\cap G_n \subset \Union_{x\in G\cap (x_0-R_0,x_0)} (x-r_{n,x},x+r_{n,x}).
\end{gathered}
\end{equation} 
Proceeding towards a contradiction suppose that 
there is a sequence $(w_i)_{i=1}^\infty$ in $G_n\cap (  x_0-R_0,x_0)\setminus \Union_{x\in G\cap (x_0-R_0,x_0)} (x-r_{n,x},x+r_{n,x})$ converging to $x_0$.
For every $i\in\N$ let $\alpha_i$ be the least ordinal number for which $y_{\alpha_i}\in(-\infty,x_0-R_0]$ and $w_i\in (y_{\alpha_i}-r_{n,\alpha_i},y_{\alpha_i}+r_{n,\alpha_i})$. 
%As $y_{\alpha_i}$s are in $(-\infty,R_0]$, the sequence $(\alpha_i)_{i=1}^\infty$ is increasing. 
By \eqref{njv} choosing a proper subsequence of $(w_i)_{i=1}^\infty$ we can assume that $(\alpha_i)_{i=1}^\infty$ is strictly increasing. 
This means that for every $i\in\N$
$$
w_i\in (y_{\alpha_i}-r_{n,y_{\alpha_i}},y_{\alpha_i}+r_{n,y_{\alpha_i}}) \setminus \Union_{\beta\in [1,\alpha_i)} (y_\beta-r_{n,y_\beta},y_\beta+r_{n,y_\beta}).
$$
Thus, if $i>1$, we have $y_{\alpha_i}\notin (y_{\alpha_{i-1}}-r_{n,y_{\alpha_{i-1}}}, y_{\alpha_{i-1}}+r_{n,y_{\alpha_{i-1}}})$ by \eqref{resze}, hence the fact
$$
y_{\alpha_{i-1}}+r_{n,y_{\alpha_{i-1}}} \in  [x_0-R_0, w_i] \subset   (y_{\alpha_i}-r_{n,y_{\alpha_i}},y_{\alpha_i}+r_{n,y_{\alpha_i}}) 
$$
implies
$$
(y_{\alpha_{i-1}}-r_{n,y_{\alpha_{i-1}}},y_{\alpha_{i-1}}+r_{n,y_{\alpha_{i-1}}})\subset (y_{\alpha_i},y_{\alpha_i}+r_{n,y_{\alpha_i}}).
$$
We obtain that $\lim_{i\to\infty} r_{n,y_{\alpha_i}} = \infty$. Furthermore, $\lim_{i\to\infty} r'_{n,y_{\alpha_i}} = \infty$ by \eqref{jo vegpontok}, which contradicts \eqref{r kicsi}.

Fix an $r\in (0,R_1)$.

If $x_0-r\in G_n$, then by \eqref{R_1} there is an $x\in (x_0-R_0,x_0)\cap G$ such that
\begin{equation}\label{metszo}
x_0-r\in (x-r_{n,x},x+r_{n,x}).
\end{equation} 
We have
\begin{equation}\label{bal_szele_0}
\begin{gathered}
|f(x+1.5r_{n,x})-f(x_0-r)| 
\underset{\text{by \eqref{x_1 x_2}}}{\le} \frac{4r_{n,x}}{50}
\le \frac{4}{25} \cdot (x+1.5r_{n,x}-(x_0-r)) \\
\end{gathered}
\end{equation}
hence
\begin{equation}\label{bal_szele_1}
\begin{gathered}
|f(x_0)-f(x_0-r)| 
\le |f(x_0)-f(x+1.5r_{n,x})|+|f(x+1.5r_{n,x})-f(x_0-r)| \\
\underset{\text{by Lemma \ref{lipnovekedes} and \eqref{bal_szele_0}}}{\le} |x_0-(x+1.5r_{n,x})|+\frac{4}{25}(x+1.5r_{n,x}-(x_0-r)). \\
\end{gathered}
\end{equation}
Moreover, $x+1.5r_{n,x} \le x_0$ as $x\in (x_0-R_0,x_0)$.  Thus
\begin{equation}\label{bal_szele}
|f(x_0)-f(x_0-r)| \underset{\text{by \eqref{bal_szele_1}}}\le r -\frac{21}{25}(x+1.5r_{n,x}-(x_0-r)).
\end{equation}
If we also have
\begin{equation}\label{szelen nagy}
|G_n\cap (x_0-r,x_0)| \le 4(x+r_{n,x}-(x_0-r)),
\end{equation}
then \eqref{bal_szele} implies
\begin{equation}\label{kicsi szelen}
\begin{split}
|f(x_0)-f(x_0-r)|
&\le r -\frac{21}{25}(x+1.5r_{n,x}-(x_0-r)) \\ 
&\le r -\frac{21}{25}(x+r_{n,x}-(x_0-r)) 
\le r -\frac{21}{100}|G_n\cap (x_0-r,x_0)|.
\end{split}
\end{equation}

If $x_0-r\not\in G_n$, or   $x_0-r\in G_n$ but  \eqref{szelen nagy} does not hold then
\begin{equation}\label{feltetel 1}
|G_n\cap (x_0-r,x_0)| \le 2\Bigg|\Union_{\substack{y\in   G , \\ (y-r_{n,y},y+r_{n,y})\subset (x_0-r,x_0)}} (y-r_{n,y},y+r_{n,y})\Bigg|.
\end{equation}
Choose finitely many points $x_1,\ldots,x_k\in G\cap (x_0-r,x_0)$ for some $k\in\N\cup\{0\}$ such that 
the intervals of the form $(x_i-r_{n,x_i},x_i+r_{n,x_i})$ are disjoint, they are subsets of $(x_0-r,x_0)$ and
\begin{equation}\label{nagy}
\begin{split}
%\Big|(G_n\cap (x-r,x)) \setminus 
\Big|\Union_{i=1}^k (x_i-r_{n,x_i},x_i+r_{n,x_i})\Big| 
\ge \frac{1}{4} \Bigg| \Union_{\substack{y\in  G, \\ (y-r_{n,y},y+r_{n,y})\subset (x_0-r,x_0)}} (y-r_{n,y},y+r_{n,y})  \Bigg|.
\end{split}
\end{equation}
Thus
\begin{equation}\label{kicsi belul}
\begin{gathered}
|f(x_0)-f(x_0-r)| \\
\underset{\text{by Lemma \ref{lipnovekedes}}}{\le} \Big|(x_0-r,x_0)\setminus\Union_{i=1}^k \Big(x_i-\frac{r_{n,x_i}}{2},x_i+\frac{r_{n,x_i}}{2}\Big)\Big| \\
+ \sum_{i=1}^k \Big|f\Big(x_i-\frac{r_{n,x_i}}{2}\Big)-f\Big(x_i+\frac{r_{n,x_i}}{2}\Big)\Big| \\
\underset{\text{by \eqref{lapos} }}{\le}  \Big|(x_0-r,x_0)\setminus\Union_{i=1}^k \Big(x_i-\frac{r_{n,x_i}}{2},x_i+\frac{r_{n,x_i}}{2}\Big)\Big| + \frac{8}{50} \Big|\Union_{i=1}^k \Big(x_i-\frac{r_{n,x_i}}{2},x_i+\frac{r_{n,x_i}}{2}\Big)\Big| \\
= r  -\frac{42}{50} \Big|\Union_{i=1}^k \Big(x_i-\frac{r_{n,x_i}}{2},x_i+\frac{r_{n,x_i}}{2}\Big)\Big| \\
\underset{\text{by \eqref{nagy} }}{\le} r-\frac{42}{400} \Bigg|\Union_{\substack{x\in G, \\ (x-r_{n,x},x+r_{n,x})\subset (x_0-r,x_0)}} (x-r_{n,x},x+r_{n,x})\Bigg| \\
\underset{\text{by \eqref{feltetel 1}}}\le r-\frac{21}{400}|G_n\cap (x_0-r,x_0)|.
\end{gathered}
\end{equation}

Thus, \eqref{kicsi belul} and \eqref{kicsi szelen} imply \eqref{bal_becsles}, which concludes the proof of the lemma.
\end{proof}

%\begin{corollary}
%There is a strongly one-sided dense $F_\sigma$ set in $\R$ which is not $\lip 1$.
%\end{corollary}

\section{Approximating closed sets with strongly one-sided dense sets}\label{not lip 1}

In \cite[Theorem 4.7]{[BHMVlip]} it was shown that $\lip 1$ sets are strongly one-sided dense and $F_\sigma$ 
(this is also an obvious consequence of Theorem \ref{char_thm} of this paper). A partial converse of this was also proved in \cite[Theorem 4.8]{[BHMVlip]} (this is a special case of Theorem \ref{char_thm} too). Nevertheless, the full converse happens to be false, as we will see. First, we need a lemma:

%During our investigations of $\lip 1$ sets we arrived at the following, fairly natural question: if $F\subseteq\mathbb{R}$ is a closed set and $\varepsilon>0$ is given, can we find a closed, strongly one-sided dense set $\widetilde{F}\subseteq F$ such that $|\widetilde{F}|>(1-\varepsilon)|F|$? 
%It turns out that the answer to that question is an emphatic ``no'' as the following theorem indicates.  

\begin{lemma}\label{unapproximable closed}
There exists a closed subset of $[0,1]$ which is of positive measure and which does not contain any nonempty, closed, SOSD subsets.
\end{lemma}

\begin{proof}
In order to make the formulation of the construction easier, we introduce the following terminology: we say that the open set $G$ is the level 0 open set, in $(a,b)$ or $[a,b]$ if $G=(a,b)$. 
The open set $G$ is a level 1 open set in $(a,b)$ or $[a,b]$, if
$$G=\left(a,a+\frac{3}{11}(b-a)\right)\cup\left(a+\frac{4}{11}(b-a),a+\frac{7}{11}(b-a)\right)\cup\left(a+\frac{8}{11}(b-a),b\right),$$
that is, if we divide $(a,b)$ into five subintervals, then $G$ is the union of the middle one, the rightmost one and the leftmost one. 
We say that these intervals are the ${\cal G}$-components of $G$, while the two closed intervals forming $(a,b)\setminus G$ are the ${\cal F}$-components of $G$. 
We use this terminology even more generally: if an open set $G\subseteq[a,b]$ is the union of finitely many open intervals, then these open intervals are the ${\cal G}$-components of $G$ in $[a,b]$, while the contiguous, nondegenerate closed intervals are the ${\cal F}$-components of $G$ in $[a,b]$.
Here, of course one can use the subspace topology of $[a,b]$ for open/closed intervals.
Analogously one can consider ${\cal G}$-components and ${\cal F}$-components of $G$ in an interval $(a,b)$.

Now we define level $k$ open sets recursively. 
We say that $G$ is the level $k$ open set in $(a,b)$ or $[a,b]$, if it can be obtained by taking the level $k-1$ open set $G_0$ in $(a,b)$, then further shrinking this set by taking 
only the union of all level 1 open sets in each of the ${\cal G}$-components of $G_0$. 
We also define the levels of the ${\cal F}$-components of a level $k$ open set $G$ in $(a,b)$ or $[a,b]$: an ${\cal F}$-component is of level $m$ if it is also an ${\cal F}$-component of the level $m$ open set in $(a,b)$, but not an ${\cal F}$-component of the level $m-1$ open set in $(a,b)$.

We define $G_{\infty}\subseteq(0,1)$ as a countable union $G_{\infty}=\bigcup_{n=1}^{\infty}G_n$, where each $G_n$ is a level $l_n$ open set and $(l_n)_{n=1}^{\infty}$ is to be chosen later. 
First, let $G_1$ be a level $l_1$ open set in $(0,1)$. 
Now in each ${\cal F}$-component $[a,b]$ of $G_1$, let us define $G_{2,(a,b)}$ as a level $l_2$ open set in $[a,b]$. 
Now let
$$G_2=\bigcup_{(a,b)}(G_{2,(a,b)}\cup\{a\}\cup\{b\}),$$
where the union runs over the ${\cal F}$-components of $G_1$. 
Now $G_2$ is almost the union of finitely many level $l_2$ open sets, except for the fact that some of its building blocks are   half-open    intervals instead of being open. 
However, $G_1\cup G_2$ is open as the union of open intervals and boundary points between such open intervals. 
Thus we can define $G_3$ similarly to $G_2$ by taking the ${\cal F}$-components of $G_1\cup G_2$ in $[0,1]$, and considering level $l_3$ open sets in each of them. 
We can continue this procedure recursively to obtain the sequence of sets $(G_n)$. 
We make precise the definition of $(l_n)_{n=1}^{\infty}$ now: this sequence is chosen such that for $F_{\infty}=[0,1]\setminus G_{\infty}$ we have
% $|F|\geq\frac{1}{2}$.
$|F_{\infty}|>0$. 
It is clear that such a choice is possible as $|G_n|\to 0$ for any fixed $n$ as $l_n\to\infty$. 
We note that $F_{\infty}$ is clearly a nowhere dense, perfect set. 
We claim that $F_{\infty}$ satisfies the statement of the lemma.

To this end, assume that $\widetilde{F}\subseteq F_\infty$ is nonempty and closed, and proceeding   towards    a contradiction, suppose that it is SOSD. 
Consequently, for all $x\in \widetilde{F}$ there exists $r_x>0$ such that for any $0<r<r_x$ the density of $\widetilde{F}$ is larger than $0.9$ in $(x-r,x)$ or $(x,x+r)$. 
Now by Baire's Category Theorem there exists an interval $(\alpha,\beta)$ and some $k\in\mathbb{N}$ such that $\left\{x:r_x>\frac{1}{k}\right\}$ is dense in $(\alpha,\beta)\cap\widetilde{F}$   and $(\alpha,\beta)\cap\widetilde{F}\neq \emptyset$. 
By shrinking this interval, if needed, we can achieve that $[\alpha,\beta]$ is an ${\cal F}$-component of $\bigcup_{n=1}^{N-1}G_n$ for some $N$, and $\beta-\alpha<\frac{1}{k}$. 
Now by construction and our hypothesis, we clearly have that $\widetilde{F}'=\widetilde{F}\cap(\alpha,\beta)$ is also nonempty, closed, and SOSD. 
Thus, it would be sufficient to arrive at a contradiction with the existence of such a set. 
Now it is clear that 
$$\widetilde{F}'\subseteq (\alpha,\beta)\setminus G_{N,(\alpha,\beta)}.$$
Assume that there exists a point $x\in\widetilde{F}'$ in a level   $l_N$,  ${\cal F}$-component of $G_{N,(\alpha,\beta)}$. 
Then by the above application of Baire's Category Theorem, $x$ can be chosen such that $r_x>\frac{1}{k}$. 
Denote the ${\cal F}$-component containing $x$ by $[p-t,p+t]$. 
Then $(p-7t,p-t)\cup(p+t,p+7t)\subseteq G_{N,(\alpha,\beta)}$ by the definition of level $l_N$ open sets. 
Thus on both sides of $x$ we can find subintervals of $(\alpha,\beta)$, notably $[x-4t,x]$ and $[x,x+4t]$ such that the density of $\widetilde{F}'$, and hence the density of $\widetilde{F}$ in each of these intervals is at most $\frac{1}{2}$, as at most one interval of length $2t$ belongs to $\widetilde{F}'$ here. 
However, as we stay inside the interval $(\alpha,\beta)$, whose length is at most $\frac{1}{k}$, one of these densities should be larger than $0.9$ by assumption. 
This gives a contradiction, thus $\widetilde{F}'$ cannot have points in a level   $l_N$,  ${\cal F}$-component of $G_{N,(\alpha,\beta)}$. 
Consequently,
$$\widetilde{F}'\subseteq (\alpha,\beta)\setminus G_{N,(\alpha,\beta)}^{(1)}$$
holds in fact, where $G_{N,(\alpha,\beta)}^{(1)}$ is the level $l_N-1$ open set in $(\alpha,\beta)$. 
Now we can repeat the argument of the previous paragraph to show that $\widetilde{F}'$ cannot have points in level   $l_N-1$,  ${\cal F}$-components of $G_{N,(\alpha,\beta)}^{(1)}$, which is equivalent to not having points in level   $l_N-1$,  ${\cal F}$-components of $G_{N,(\alpha,\beta)}$. 
Proceeding by induction, we can show for any $m=1,2,...,l_N$ that
$$\widetilde{F}'\subseteq (\alpha,\beta)\setminus G_{N,(\alpha,\beta)}^{(m)}$$
holds, where $G_{N,(\alpha,\beta)}^{(  m  )}$ is the level $l_N-m$ open set in $(\alpha,\beta)$. 
However, the $m=l_N$ case means that $\widetilde{F}'$ does not have any points in $(\alpha,\beta)$, that is,  $\widetilde{F}'$ is empty. 
This gives a contradiction, which concludes the proof.
\end{proof}

%Next we discuss a consequence of Theorem \ref{unapproximable closed}

\begin{theorem}\label{bad SOSD}
There  exists  an SOSD $F_\sigma$ set which does not contain any nonempty, closed, SOSD subsets.
\end{theorem}

\begin{proof}
By Lemma \ref{unapproximable closed} we can take countably many closed sets $( F_n )_{n=1}^\infty$ such that
they do not contain any nonempty, closed, SOSD subsets,
they are pairwise disjoint and
their union $F$ is of full measure in $\R$.
Thus $F$ is SOSD and $F$ is $F_\sigma$.

Let $F'$ be a  nonempty, closed  SOSD subset of $F$. 
Set $F'_n := F'\cap F_n$ for every $n\in\N$.
It is clear that the $F'_n$s are nowhere dense and none of them contains a nonempty, closed, SOSD set.
This implies that those points of $F'_n$ at which $F'_n$ is not SOSD form a dense subset of $F'_n$.

We define sequences $(x_n)_{n=1}^\infty$ in $\R$, $(m_n)_{n=1}^\infty$ in $\N$ and closed intervals $(I_n)_{n=1}^\infty$ such that $x_n\in F'_{m_n}$, the set $F'_{m_n}$ is not SOSD at $x_n$ and $\inte(I_n)$, the interior of $I_n$ is a neighbourhood of $x_n$.
Set $m_1:=1$, take an $x_1\in F'_1$ such that $F'_1$ is not SOSD at $x_1$ and let $I_1 := [x_1-1,x_1+1]$.
We proceed by recursion.
Suppose that $n>1$ and we have defined $m_i$, $x_i$ and $I_i$    so that  for every  $i \in  \{1,\ldots,n-1\}$ we have
\begin{align*}
 x_i \in F_{m_i}',  &\ F'_{m_i} \mbox{ is not SOSD at }x_i , \mbox{ diam }(I_i) \leq  2/i,  \text{ and}\\
&  x_i \in \mbox{int}(I_i)\subset I_{i-1}, \ \ \text{ when }i>1.
\end{align*}
As $\Union_{k=1}^{m_{n-1}-1} F'_k$ is closed and it does not contain $x_{n-1}$, 
we can take a closed interval $I_n\subset I_{n-1}$  such that $x_{n-1}\in \inte(I_{n})$,
\begin{equation}\label{away}
d\big(I_n,\Union\nolimits_{k=1}^{m_{n-1}-1} F'_k\big) > 0
\end{equation}
and $\diam (I_n)  \leq 2/n $. 
Since $F'$ is SOSD, but $F'_{m_{n-1}}$ is not SOSD at $x_{n-1}$, there is an $m_n\in\N\cap (m_{n-1},\infty)$ for which $|F'_{m_n}\cap I_n|>0$.
 Using the fact that those points of $F'_{m_n}$ at which $F'_{m_n}$ is not SOSD form a dense subset of $F'_{m_n}$ 
we can take an $x_n\in F'_{m_n}\cap \inte(I_{n-1})$ such that $F'_{m_n}$ is not SOSD at $x_n$.

As $x_n\in I_n$ for every $n\in\N$ and $\lim_{n\to\infty} \diam(I_n) = 0$, we obtain that there is a unique element $x^*$ of $\bigcap_{n=1}^\infty I_n$ and $\lim_{n\to\infty} x_n = x^*$. By \eqref{away}, we have $x^*\notin F'$. This implies that $F'$ is not closed, which concludes the proof.
\end{proof}

Let us observe the obvious fact that the $F_\sigma$ set guaranteed by the above theorem cannot be written as the union of countably many SOSD closed sets. Paired with Theorem \ref{char_thm},  this  immediately implies the following corollary:
\begin{corollary}
  There exists an SOSD $F_\sigma$ set which is not $\lip 1$.
\end{corollary}

\end{document}